\documentclass[12pt]{amsart}
\usepackage[usenames]{color}
\usepackage{times}
\usepackage{amsfonts}
\usepackage{amsthm}
\usepackage{amsmath}
\usepackage{psfrag}
\usepackage{times}
\usepackage{latexsym,color}
\usepackage{amssymb}
\usepackage{graphicx}
\usepackage{epsfig}
\usepackage{gastex}
\advance \topmargin by -\headheight
\advance \topmargin by -\headsep
\evensidemargin \oddsidemargin
\newcommand{\Fin}{\mbox{Fin}(\nats)}

\newcommand{\R}{\mathbb{R}}
\newcommand{\N}{\mathbb{N}}

\newcommand{\Z}{\mathbb{Z}}

\newcommand{\larr}{\left( \begin{array}{c}}
\newcommand{\rarr}{\end{array} \right) }

\newcommand{\lsqarr}{\left[ \begin{array}{c}}
\newcommand{\rsqarr}{\end{array} \right]}

\makeatletter
\def\Ddots{\mathinner{\mkern1mu\raise\p@
\vbox{\kern7\p@\hbox{.}}\mkern2mu
\raise4\p@\hbox{.}\mkern2mu\raise7\p@\hbox{.}\mkern1mu}}
\makeatother

\newcommand{\nats}{{\mathbb N}}
\newcommand{\ints}{{\mathbb Z}}

\newtheorem{thm}{Theorem}
\newtheorem{cor}{Corollary}

\newtheorem{theorem}{Theorem}[section]

\newtheorem{lemma}[theorem]{Lemma}
\newtheorem{corollary}[theorem]{Corollary}
\newtheorem{proposition}[theorem]{Proposition}

\newtheorem{remark}[theorem]{Remark}

\newtheorem{definition}[theorem]{Definition}
\theoremstyle{definition}

\begin{document}

\title[Central sets and substitutive dynamical systems]{Central sets and substitutive dynamical systems}

\author[M. Barge]{Marcy Barge}
\address{Department of Mathematics\\
Montana State University\\
Bozeman, MT 59717-0240, USA.}
\email{barge@math.montana.edu}

\author[L.Q. Zamboni]{Luca Q. Zamboni}
\address{Universit\'e de Lyon\\
Universit\'e Lyon 1\\
CNRS UMR 5208\\
Institut Camille Jordan\\
43 boulevard du 11 novembre 1918\\
F69622 Villeurbanne Cedex, France. Also,
Department of Mathematics\\
FUNDIM\\
University of Turku\\
FIN-20014 Turku, Finland.}
\email{lupastis@gmail.com}

\date{April 30, 2012}

\maketitle

\begin{abstract}
In this paper we establish a new connection between central sets and the  {\it strong coincidence conjecture} for fixed points of irreducible primitive substitutions of Pisot type. Central sets, first introduced by  Furstenberg using notions from topological dynamics, constitute a special class of subsets of $\nats$ possessing strong combinatorial properties: Each central set contains arbitrarily long arithmetic progressions, and solutions to all partition regular systems of homogeneous
linear equations. We  give an equivalent reformulation of the strong coincidence condition in terms of central sets and minimal idempotent ultrafilters in the Stone-\v Cech compactification $\beta \nats .$
This provides a new arithmetical approach to an outstanding conjecture in tiling theory, the {\it Pisot substitution conjecture.}  The results in this paper rely on  interactions between different areas of mathematics, some of which had not previously been directly linked: They include the general theory of combinatorics on words,  abstract numeration systems, tilings, topological dynamics and the algebraic/topological properties of  Stone-\v Cech compactification of $\nats.$
\end{abstract}

\section{Introduction}

An important open problem in the theory of substitutions is the so-called {\it strong coincidence conjecture}: It states that each pair of fixed points $x$ and $y$ of an irreducible primitive substitution of Pisot type are {\it strongly coincident}: There exist a letter $a$ and a pair of Abelian equivalent words $s,t,$  such that $sa$ is a prefix of $x$ and $ta$ is a prefix of $y.$
This combinatorial condition, originally due to P. Arnoux and S. Ito, is an extension of a similar condition considered by F.M. Dekking in \cite{Dek} in the case of uniform substitutions. In this case Dekking proves that the condition is satisfied by the ``pure base" of the substitution if and only if  the associated substitutive subshift has {\it pure discrete spectrum}, i.e., is metrically isomorphic with translation on a compact Abelian group. The strong coincidence conjecture has been verified for irreducible primitive substitutions of Pisot type on a binary alphabet in \cite{BD} and is otherwise still open.

The strong coincidence conjecture is linked to diffraction properties of one-dimensional atomic arrangements in the following way. It is shown in \cite{Dw} and \cite{LMS} that
an atomic arrangement determined by a substitution has pure point diffraction spectrum
(i.e., is a perfect quasicrystal) if and only if the tiling system associated with the substitution has pure discrete dynamical spectrum. The {\em Pisot substitution conjecture} asserts that the dynamical spectrum of the tiling system associated with an irreducible Pisot substitution has pure discrete dynamical spectrum. For the latter to hold, it is necessary, and conjecturally sufficient, for the substitution to satisfy the strong coincidence condition.

In this paper we establish a link between
the strong coincidence conjecture and central sets, originally introduced by Furstenberg in \cite{F}. More precisely, we obtain an equivalent reformulation of the conjecture in terms of minimal idempotents in the Stone-\v Cech compactification $\beta \nats .$

Let $\nats =\{0,1,2,3,\ldots\}$ denote the set of natural numbers, and $\Fin$ the set of all non-empty finite subsets of $\nats.$ A subset $A$ of $\nats$ is called an {\it IP-set} if $A$ contains $\{\sum _{n\in F}x_n \,| \,F\in \Fin\}$ for some infinite sequence of natural numbers $x_0<x_1<x_2 \cdots .$ A subset $A\subseteq \nats$  is called an {\it IP$^*$-set} if $A\cap B\neq \emptyset $ for every IP-set $B\subseteq \nats.$
 In \cite{F}, Furstenberg introduced a special class of IP-sets, called central sets, having a substantial combinatorial structure.  Central sets were originally defined in terms of  topological dynamics:

\begin{definition}\label{Cen2} A subset $A\subset \nats$ is called {\it central} if there exists a compact metric space $(X,d)$ and a continuous map $T: X\rightarrow X,$ points $x,y \in X$ and a neighborhood $U$ of $y$ such that
\begin{itemize}
\item $y$ is a uniformly recurrent point in $X,$
\item $x$ and $y$ are proximal,
\item $A=\{n \in \nats\,|\, T^n(x)\in U\}.$
\end{itemize}
 We say $A\subset \nats$ is {\it central}$^*$ if  $A\cap B\neq \emptyset $ for every central set $B\subseteq \nats.$
\end{definition}

\noindent Recall that $x$ is said to be {\it uniformly recurrent} in $X$ if for every neighborhood $V$ of $x$ the set
$\{n\,|\, T^n(x)\in V\}$ is syndetic, i.e., of bounded gap. Two points $x,y\in X$ are said to be {\it proximal} if for every $\epsilon >0$ there exists $n\in \nats$ such that $d(T^n(x),T^n(y))<\epsilon.$   

It is not evident from the above definition that central sets are IP-sets. 
The connection between the two lies in the algebraic and topological properties of the Stone-\v Cech compactification $\beta \nats .$ We regard $\beta \nats$ as the collection of all ultrafilters on $\nats.$ 
There is a natural extension of the operation of addition $+$ on $\nats$ to $\beta \nats$ making $\beta \nats$ a compact {\it left-topological semigroup.}
Via a celebrated result of Ellis \cite{E}, $\beta \nats$  contains {\it idempotents}, i.e.,  ultrafilters $p\in \beta \nats$  satisfying  $p+p=p.$

A striking result due to Hindman links IP-sets and idempotents in $\beta \nats:$ A subset $A\subseteq \nats$ is an IP-set if and only if $A\in p$ for some idempotent $p\in \beta \nats$ (see Theorem 5.12 in \cite{HS}). Thus $A$ is an IP$^*$-set if and only if $A\in p$ for every idempotent $p\in \beta \nats$ (see Theorem~2.15 in \cite{VB2}). \\
It follows that given any finite partition of an IP-set, at least one element of the partition is again an IP-set. In other words the property of being an IP-set is {\it partition regular}, i.e., cannot be destroyed via a finite partitioning. Other examples of partition regularity are given by the pigeonhole principle, sets having positive upper density, and sets having arbitrarily long arithmetic progressions (Van der Waerden's theorem).

In \cite{BH}, Bergelson and Hindman  showed that central sets  too may alternatively be defined in terms of  a special class of  ultrafilters, called {\it minimal idempotents}. Every compact Hausdorff left-topological semigroup $\mathcal{S}$  admits a smallest two sided ideal $K(\mathcal{S})$ which is at the same time the union of all minimal right ideals of $\mathcal{S}$ and the union of all minimal left ideals of $\mathcal{S}$ (see for instance \cite{HS}).
It is readily verified that the intersection of any minimal left ideal with any minimal right ideal is a group. In
particular, there are idempotents in $K(\mathcal{S}).$ Such idempotents are called minimal and their elements are called central sets, i.e.,  $A \subset \nats$ is a {\it central set} if it is a member of some minimal idempotent in $\beta \nats.$ 

 It now follows that every central set is an IP-set and that the
property of being central is partition regular. Central sets are
known to have substantial combinatorial structure. For example,
any central set contains arbitrarily long arithmetic progressions,
and solutions to all partition regular systems of homogeneous
linear equations (see for example \cite{BHS}).

An ultrafilter may be thought of as a $\{0,1\}$-valued finitely additive probability measure defined on all subsets of $\nats.$ This notion of measure induces a notion of convergence ($p$-$\lim_n)$ for sequences indexed by $\nats,$ which we regard as a mapping  from words to words. This key notion of convergence allows us to reformulate  the strong coincidence conjecture in terms of central sets:

\begin{thm}\label{SCCP} Let $\tau$ be an irreducible primitive substitution of Pisot type. Then for any pair of fixed points $x$ and $y$ of $\tau$ the following are equivalent:
\begin{enumerate}
\item $x$ and $y$ are strongly coincident.
\item There exists a minimal idempotent $p\in \beta \nats$ such that $y=p$-$\lim_n T^n(x)$ where $T$ denotes the shift map.
\item For any prefix $u$ of $y,$ the set of occurrences of  $u$ in $x$  is a central set.
\end{enumerate}
\end{thm}

\noindent Theorem~\ref{SCCP} asserts that for an irreducible primitive substitution of Pisot type, the strong coincidence condition is equivalent to the condition that the idempotent ultrafilters in $\beta \nats$ permute the fixed points of the substitution.  

In the context of uniformly recurrent words, IP-sets and central sets are one and the same (see Theorem~\ref{same} proved in  \cite {BPZ}); thus we obtain:

\begin{cor}\label{scip} Let $\tau$ be an irreducible primitive substitution of Pisot type. Then for any pair of fixed points $x$ and $y$ of $\tau,$ $x$ and $y$ are strongly coincident if and only if for any prefix $u$ of $y,$ the set of occurrences of  $u$ in $x$  is an IP-set.
\end{cor}

\noindent Since IP-sets may be defined arithmetically in terms of finite sums of distinct terms of infinite sequences $(x_n)_{n\in \nats} $ of natural numbers, Corollary~\ref{scip}  provides an arithmetical approach to solving the strong coincidence conjecture. To this end, 
we show that certain abstract numeration systems first introduced by 
J.-M. Dumont and A. Thomas in \cite{DT1, DT2} provide a useful arithmetic tool to the conjecture.\\

\paragraph{\bf Acknowledgements}
The second author is supported in part by a FiDiPro grant from the Academy of Finland.

\section{Words and substitutions}

Given a finite non-empty set ${\mathcal A}$ (called the {\it alphabet}), we denote by ${\mathcal A}^*,$  ${\mathcal A}^\nats$  and ${\mathcal A}^\ints$ respectively the set of finite words,  the set of (right) infinite words, and the set of bi-infinite words over the alphabet ${\mathcal A}$. Given a finite word $u =a_1a_2\ldots a_n$ with $n \geq 1$ and $a_i \in {\mathcal A},$ we denote the length $n$ of $u$ by $|u|.$ The  \textit{empty word} will be denoted by $\varepsilon$ and we set $|\varepsilon|=0.$ We put $ {\mathcal A}^+= {\mathcal A}^*-\{\varepsilon\}.$ For each $a\in {\mathcal A},$ we let $|u|_a$  denote the number of occurrences of the letter $a$ in $u.$
Two words $u$ and $v$ in $A^*$ are said to be {\it Abelian equivalent,} denoted $u\sim_{\mbox{ab}} v,$  if and only if $|u|_a=|v|_a$ for all $a\in {\mathcal A}.$ It is readily verified that $\sim_{\mbox{ab}}$ defines an equivalence relation on ${\mathcal A}^*.$

Given an infinite word
$\omega \in {\mathcal A}^\nats,$ a word $u\in {\mathcal A}^+$ is called a  {\it factor} of $\omega$
if  $u=\omega_{i}\omega_{i+1}\cdots \omega_{i+n}$ for some natural numbers $i$ and $n.$
We denote by ${\mathcal F}_{\omega}(n)$ the set of all factors of $\omega$ of length $n,$ and set
\[{\mathcal F}_{\omega} =\bigcup _{n\in \nats} {\mathcal F}_{\omega}(n).\]
For each finite word $u$ on the alphabet ${\mathcal A}$ we set
\[\omega\big|_{u}=\{ n\in \nats \,|\, \omega_n\omega_{n+1}\ldots \omega_{n+|u|-1}=u\}.\]
In other words, $\omega\big|_u$ denotes the set of all occurrences of $u$ in $\omega.$

We say $\omega$ is {\it recurrent} if for every $u\in {\mathcal F}_{\omega}$  the set $\omega\big|_u$ is infinite.
We say $\omega$ is {\it uniformly recurrent} if for every $u\in {\mathcal F}_{\omega}$  the set $\omega\big|_u$ is syndedic, i.e., of bounded gap.

We endow ${\mathcal A}^\nats$ with the topology generated by the metric
\[d(x, y)=\frac 1{2^n}\,\,\mbox{where} \,\, n=\inf\{k :x_k\neq y_k\}\] %replaced min by inf
whenever $x=(x_n)_{n\in \nats}$ and $y=(y_n)_{n\in \nats}$ are two elements of ${\mathcal A}^\nats.$ Let $T:{\mathcal A}^\nats \rightarrow {\mathcal A}^\nats$ denote the {\it shift} transformation defined by $T: (x_n)_{n\in \nats}\mapsto (x_{n+1})_{n\in \nats}.$ By a {\it subshift} on ${\mathcal A}$ we mean a pair $(X,T)$ where $X$ is a closed and $T$-invariant subset of ${\mathcal A}^\nats.$ A subshift $(X,T)$ is said to be {\it minimal}
whenever $X$ and the empty set are the only $T$-invariant closed subsets of $X.$ To each $\omega \in {\mathcal A}^\nats$ is associated the subshift $(X,T)$ where $X$ is the shift orbit closure of $\omega.$ If $\omega$ is uniformly recurrent, then the associated subshift $(X,T)$ is minimal.
Thus any two words $x$ and $y$ in $X $ have exactly the same set of factors, i.e., ${\mathcal F}_x={\mathcal F}_y.$ In this case we denote by ${\mathcal F}_{X}$ the set of factors of any word $x\in X.$

Two points $x ,y$ in $X$ are said to be {\it proximal} if and only if for each $N>0$ there exists $n\in \nats$ such that  \[x_nx_{n+1}\ldots x_{n+N}= y_ny_{n+1}\ldots y_{n+N}.\]
%Two points $x,y\in X$ are said to be {\it regionally proximal} if for every prefix $u$ of $x$ and $v$ of $y,$ there exist points $x',y'\in X$ with $x'$ beginning in $u$ and $y'$ beginning in $v$ and with $x'$ proximal to $y'.$ Clearly if two points in $X$ are proximal, then they are regionally proximal.
A point $x\in X$ is called {\it distal} if the only point in $ X$ proximal to $x$ is $x$ itself.  A minimal subshift $(X,T)$
is said to be {\it topologically mixing} if for every any pair of factors $u,v \in {\mathcal F}_X$ there exists a positive integer $N$ such that for each $n\geq N,$ there exists a block of the form $uWv \in  {\mathcal F}_X$ with $|W|=n.$
A minimal subshift $(X,T)$
is said to be {\it topologically weak mixing} if for every pair of factors $u,v \in {\mathcal F}_X$
the set
\[\{n \in \nats\,|\, u{\mathcal A}^nv \cap {\mathcal F}_X \neq \emptyset\}\]
is thick, i.e., for every positive integer $N,$ the set contains $N$ consecutive positive integers.\\

A {\it substitution} $\tau $ on an alphabet $ {\mathcal A}$ is
a mapping $\tau : {\mathcal A}\rightarrow  {\mathcal A}^+.$
The mapping
$\tau $  extends by concatenation to maps (also
denoted $\tau )$
$ {\mathcal A} ^*\rightarrow  {\mathcal A} ^*$ and $ {\mathcal A}^{\nats}\rightarrow  {\mathcal A}^{\nats}.$
The {\it Abelianization} of $\tau$ is the square matrix $M_\tau$ whose $ij$-th entry is equal to $|\tau(j)|_i,$ i.e., the number of occurrences of $i$ in $\tau(j).$
A substitution $\tau $ is said to be  {\it primitive} if
there is a positive integer $n$ such that
for each pair $(i,j)\in  {\mathcal A} \times  {\mathcal A},$
the  letter $i$ occurs in $\tau
^n(j).$  Equivalently if all the entries of $M_\tau^n$ are strictly positive. In this case it is well known that the matrix $M_\tau$ has a simple positive Perron-Frobenius eigenvalue called the {\it dilation} of $\tau.$ A substitution $\tau$ is said to be {\it irreducible} if the minimal polynomial of its dilation is equal to the characteristic polynomial of its Abelianization $M_\tau.$  A substitution $\tau$  is
said to be of {\it Pisot type} if its dilation is a Pisot number. Recall that a Pisot number is an algebraic integer greater than $1$ all of whose algebraic conjugates lie strictly inside the unit circle.

Let $\tau $ be a primitive substitution on $ {\mathcal A}.$
A word $\omega \in  {\mathcal A}^{\nats}$ is called a {\it fixed point} of $\tau$
if $\tau (\omega)=\omega,$ and is called a {\it periodic point} if $\tau ^m(\omega)=\omega$ for
some $m>0.$
Although $\tau $ may fail to have a fixed point, it has at least one periodic point.
Associated to $\tau $ is the topological dynamical system
$(X,T),$ where $X$ is the shift orbit closure of a periodic point $\omega$ of $\tau.$
The primitivity of $\tau $ implies that $(X,T)$ is
independent of the choice of periodic point and is  minimal.

\section{Ultrafilters, IP-sets and central sets}

\subsection{Stone-\v Cech compactification}

The Stone-\v Cech compactification $\beta \nats$ of $\nats$ is one of many compactifications of $\nats.$ It is in fact the largest  compact Hausdorff space generated by $\nats.$ More precisely $\beta \nats$ is a compact and Hausdorff space together with a continuous injection $i:\nats \hookrightarrow \beta \nats$ satisfying the following universal property: any continuous map $f:\nats \rightarrow X$
into a compact Hausdorff space $X$ lifts uniquely to a continuous map $\beta f: \beta \nats \rightarrow X,$ i.e., $f=\beta f\circ i.$ This universal property characterizes $\beta \nats$ uniquely up to homeomorphism.
While there are different methods for constructing the Stone-\v Cech compactification of $\nats,$   we shall regard $\beta \nats$  as the set of all ultrafilters on $\nats$ with the {\it Stone topology.}

\vspace{.1in}

Recall that a set $\mathcal U$ of subsets of $\nats$ is called an {\it ultrafilter} if the following conditions hold:
\begin{itemize}
\item $\emptyset \notin \mathcal U.$
\item If $A\in \mathcal U$ and $A\subseteq B,$ then $B\in \mathcal U.$
\item $A\cap B \in \mathcal U$ whenever both $A$ and $B$ belong to $\mathcal U.$
\item For every $A\subseteq \nats$ either $A\in \mathcal U$ or $A^c\in \mathcal U$ where $A^c$ denotes the complement of $A.$
\end{itemize}

For every natural number $n\in \nats,$ the set $\mathcal{U}_n=\{A\subseteq \nats \,|\, n\in A\}$ is an example of an ultrafilter. This defines an injection $i:\nats \hookrightarrow \beta \nats$ by: $n\mapsto \mathcal{U}_n.$
An ultrafilter of this form is said to be {\it principal.} By way of Zorn's lemma, one can show the existence of  non-principal (or {\it free}) ultrafilters.

It is customary to denote elements of $\beta \nats$ by letters $p,q,r \ldots.$
For each set $A \subseteq \nats,$ we set $A^\circ =\{p\in \beta \nats | A\in p\}.$ Then the set $\mathcal{B}=\{A^\circ | A\subseteq \nats\}$ forms a basis for the open sets (as well as a basis for the closed sets) of $\beta \nats$  and defines a topology on $\beta \nats$ with respect to which $\beta \nats$ is  both compact and Hausdorff.\footnote{Although the existence of free ultrafilters requires Zorn's lemma, the cardinality of $\beta \nats$ is $2^{2^\nats}$ from which it follows that $\beta \nats$ is not metrizable.} It is not difficult to see that the injection $i:\nats \hookrightarrow \beta \nats$ is continuous and satisfies the required universal property. In fact, given a continuous map $f: \nats \rightarrow X$ with $X$ compact Hausdorff, for each ultrafilter $p\in \beta \nats,$ the pushfoward $f(p)=\{f(A)\,|\, A\in p\}$ generates a unique ultrafilter $\beta f(p)$ on $X.$

There is a natural extension of the operation of addition $+$ on $\nats$ to $\beta \nats$ making $\beta \nats$ a compact {\it left-topological semigroup.} More precisely addition of two ultrafilters $p,q$ is defined by the following rule:

\[p+q=\{A \subseteq \nats \,|\, \{n\in \nats | A-n \in p\}\in q\}.\]

It is readily verified that $p+q$ is once again an ultrafilter and that
for each fixed $p\in \beta \nats,$ the mapping $q\mapsto p+q$ defines a continuous map from $\beta \nats $ into itself.\footnote{Our definition of addition of ultrafilters is the same as that given in \cite{VB2} but is the reverse of that given in \cite{HS} in which $A\in p+q$ if and only if $\{n\in \nats | A-n \in q\}\in p\}.$ In this case, $\beta \nats$ becomes a compact right-topological semigroup.} The operation of addition in $\beta \nats$ is associative and for principal ultrafilters we have $\mathcal{U}_m + \mathcal{U}_n= \mathcal{U}_{m+n}.$ However in general addition of ultrafilters is highly non-commutative. In fact it can be shown that the center is precisely the set of all principal ultrafilters \cite{HS}.

\subsection{IP-sets and central sets}

Let $(\mathcal{S}, +)$ be a semigroup. An element $p\in \mathcal{S}$  is called an {\it idempotent} if $p+p=p.$
We recall the following result of Ellis \cite{E}:

\begin{theorem}[Ellis \cite{E}]\label{Ellis} Let $(\mathcal{S}, +)$ be a compact left-topological semigroup (i.e., $\forall x\in \mathcal{S}$ the mapping $y\mapsto x+y$ is continuous). Then $\mathcal{S}$ contains an idempotent.
\end{theorem}

\noindent It follows that $\beta \nats$ contains a non-principal ultrafilter $p$ satisfying $p+p=p.$ In fact, we could simply apply Ellis's result to the semigroup $\beta \nats - \mathcal{U}_0.$ This would then exclude the only principal idempotent ultrafilter, namely $\mathcal{U}_0.$ From here on, by an idempotent ultrafilter in $\beta \nats$ we mean a free idempotent ultrafilter.  The following striking result due to Hindman establishes a link between IP-sets and idempotent ultra filters:

\begin{theorem}[Theorem 5.12 in \cite{HS}]\label{Hind} A subset $A\subseteq \nats$ is an IP-set if and only if $A\in p$ for some idempotent $p\in \beta \nats.$
\end{theorem}

A subset $\mathcal{I}\subseteq \mathcal{S}$ is called a {\it right (resp. left) ideal} if
$\mathcal{I}+\mathcal{S}\subseteq \mathcal{I}$
 (resp. $\mathcal{S}+\mathcal{I}\subseteq \mathcal{I}$). It is called a {\it two sided ideal} if it is both a left and right ideal.
A right (resp. left) ideal $\mathcal{I}$ is called {\it minimal} if every right (resp. left) ideal $\mathcal{J}$ included in $\mathcal{I}$ coincides with $\mathcal{I}.$

Minimal right/left ideals do not necessarily exist  e.g. the commutative semigroup  $(\nats, +)$ has no minimal right/left ideals (the ideals in $\nats$ are all of the form $\mathcal{I}_n=[n, +\infty)=\{m\in \nats\,|\,m\geq n\}.)$ However,
every compact Hausdorff left-topological semigroup $\mathcal{S}$ (e.g., $\beta \nats)$  admits a smallest two sided ideal $K(\mathcal{S})$ which is at the same time the union of all minimal right ideals of $\mathcal{S}$ and the union of all minimal left ideals of $\mathcal{S}$ (see for instance \cite{HS}).
It is readily verified that the intersection of any minimal left ideal with any minimal right ideal is a group. In
particular, there are idempotents in $K(\mathcal{S}).$ Such idempotents are called minimal and their elements are called central sets:

\begin{definition} An idempotent $p$  is called a {\it minimal} idempotent of $\mathcal{S}$ if it belongs to $K(\mathcal{S}).$
\end{definition}

\begin{definition}\label{Cen1} A subset $A \subset \nats$ is called {\it central} if it is a member of some minimal idempotent in $\beta \nats.$ It is called a central$^*$-set if it belongs to every minimal idempotent in $\beta \nats.$
\end{definition}

The equivalence between definitions \ref{Cen2} and \ref{Cen1} is
due to Bergelson and Hindman in \cite{BH}. It follows from the
above definition that every central set is an IP-set and that the
property of being central is partition regular. Central sets are
known to have substantial combinatorial structure. For example,
any central set contains arbitrarily long arithmetic progressions,
and solutions to all partition regular systems of homogeneous
linear equations (see for example \cite{BHS}). Many of the rich
properties of central sets are a consequence of the {\it Central
Sets Theorem} first proved by Furstenberg in Proposition~8.21 in
\cite{F} (see also \cite{DHS, BHS, HS2}). Furstenberg pointed out
that as an immediate consequence of the Central Sets Theorem one
has that whenever $\nats$  is divided into finitely many classes,
and a sequence $(x_n)_{n\in \nats}$ is given, one of the classes
must contain arbitrarily long arithmetic progressions whose
increment belongs to $\{\sum _{n\in F}x_n | F\in \Fin\}.$

\subsection{Limits of ultrafilters}
It is often convenient to think of an ultrafilter $p$ as a $\{0,1\}$-valued, finitely additive probability measure on the power set of $\nats.$ More precisely, for any subset $A\subseteq \nats,$ we say $A$ has $p$-measure $1,$ or is $p$-large if $A\in p.$ This notion of measure gives rise to a notion of convergence of sequences indexed by $\nats$ denoted $p$-$\lim_n.$
From our point of view, it is  more natural to consider it as a mapping $p^*$ from words to words. More precisely, let $\mathcal{A}$ denote a non-empty finite set. 

\begin{definition}\label{p*} For each $p\in \beta \nats$ and $\omega\in \mathcal{A}^\nats,$ we define $p^*(\omega) \in\mathcal{A}^\nats$ by the condition:   $u\in \mathcal{A}^*$ is a prefix of $p^*(\omega) $ $\Longleftrightarrow$ $\omega \big|_u\in p.$
\end{definition}

\noindent We note that if  $u,v\in \mathcal{A}^*, $ $\omega \big|_u, \omega\big|_v \in p$  and  $|v|\geq |u|,$  then $u$ is a prefix of $v.$ In fact, if $v'$ denotes the
prefix of $v$ of length $|u|$ then as $\omega \big|_v \subseteq \omega \big|_{v'}, $ it follows that $\omega \big|_{v'}\in p$ and hence $u=v'.$
Thus $p^*(\omega)$ is well defined.
It follows immediately from the definition of $p^*,$ Definition~\ref{Cen1}  and Theorem~\ref{Hind} that

\begin{lemma}\label{IP} The set $\omega\big|_u$ is an IP-set (resp. central set) if and only if $u$ is a prefix of $p^*(\omega)$ for some idempotent (resp. minimal idempotent) $p\in \beta \nats.$
\end{lemma}

It is readily verified that our definition of $p^*$ coincides with that of $p$-$\lim_n.$
More precisely, given a sequence $(x_n)_{n\in \nats}$ in a topological space and an ultrafilter $p\in \beta \nats,$  we write
$p$-$\lim_nx_n=y$ if for every neighborhood $U_y$ of $y$ one has $\{n\,| x_n \in U_y\}\in p.$
In our case we have
$p^*(\omega)=p$-$\lim_n(T^n(\omega))$  (see \cite{H2}). With this in mind, we obtain (see for instance \cite{Bl,H2}):

\begin{lemma}\label{idempotent} Let $p,q \in \beta \nats$ and $\omega\in \mathcal{A}^\nats.$ Then
\begin{itemize}
\item $p^*(T(\omega))=T(p^*(\omega))$ where $T:\mathcal{A}^\nats \rightarrow \mathcal{A}^\nats$ denotes the shift map.
\item $(p+q)^*(\omega)=q^*(p^*(\omega)).$ In particular, if $p$ is an idempotent, then $p^*(p^*(\omega))=p^*(\omega).$
\end{itemize}
\end{lemma}

We will make use of the following key result in \cite{HS} (see also Theorem 1 in \cite{Bl} and Theorem 3.4 in \cite{VB2}):

 \begin{theorem}[Theorem 19.26 in \cite{HS}]\label{Berg} Let $(X,T)$ be a topological dynamical system. Then if two points $x,y \in X$ are proximal  with $y$ uniformly recurrent, then  there exists a minimal idempotent $p\in \beta \nats$ such that $p^*(x)=y.$
 \end{theorem}

\section{Strong coincidence condition}
Let $n\geq 2$ be a positive integer  and set ${\mathcal A}=\{1,2,\ldots, n\}.$ A primitive substitution  $\tau: {\mathcal A}\rightarrow {\mathcal A}^+$ is said to satisfy the {\it strong coincidence condition} if and only if  any pair of fixed points $x$ and $y$ are {\it strongly coincident}, i.e.,  we can write $x=scx',$ and $y=tcy'$ for some $s,t \in {\mathcal A}^+,$ $c\in {\mathcal A},$  and $x',y' \in {\mathcal A}^\infty$ with
$s\sim_{\mbox{ab}} t.$  This combinatorial condition, originally due to P. Arnoux and S. Ito, is an extension of a similar condition considered by F.M. Dekking in \cite{Dek} in the case of {\it constant length substitutions}, i.e., when $|\tau(a)|=|\tau(b)|$ for all $a,b\in {\mathcal A}.$ Every such substitution $\tau$ has an algorithmically determined ``pure base" substitution and Dekking proves that the strong coincidence condition is satisfied by the pure base if and only if  the substitutive subshift associated with $\tau$ has {\it pure discrete spectrum}, i.e., is metrically isomorphic with translation on a compact Abelian group.
The Thue-Morse substitution is equal to its pure base and clearly does not satisfy the strong coincidence condition - in fact the two fixed points disagree in each coordinate.
It is conjectured however that if $\tau$ is an {\em irreducible} primitive  substitution of Pisot type,  then $\tau$ satisfies the strong coincidence condition.
This conjecture is established for binary primitive substitutions of Pisot type in \cite{BD}. Otherwise the conjecture remains open for substitutions defined on alphabets of size greater than two.  Substitutions of Pisot type provide a framework for non-constant length substitutions in which the strong coincidence condition is necessary (and, conjecturally, sufficient) for pure discrete spectrum (see \cite{BD, BD2, BK1,BK2}). We now establish the following reformulation of the strong coincidence condition in terms of central sets:

\begin{theorem}\label{SCCP2} Let $\tau$ be an irreducible primitive substitution of Pisot type. Then for any pair of fixed points $x$ and $y$ of $\tau$ the following are equivalent:
\begin{enumerate}
\item $x$ and $y$ are strongly coincident.
\item $x$ and $y$ are proximal.
\item There exists a minimal idempotent $p\in \beta \nats$ such that $y=p^*(x).$
\item For any prefix $u$ of $y,$ the set $x\big|_u$ is a central set.
\end{enumerate}
\end{theorem}

\begin{remark}\rm {For a general primitive substitution we always have that $ (1)\Longrightarrow (2) \Longrightarrow (3) \Longrightarrow (4).$ But in general in the non-Pisot setting, these conditions need not be equivalent: For instance,
the two fixed points of the uniform substitution $a\mapsto aaab,$ $b\mapsto bbab$ are proximal but do not satisfy the strong coincidence condition.  V. Bergelson and Y. Son  \cite{BergCom} showed that the fixed points of $a\mapsto aab,$ $b\mapsto bbaab$ satisfy (4) but not (1), (2) and (3). It would be interesting to understand in general under what conditions do the idempotent ultrafilters permute the fixed points of substitutions.}
\end{remark}

\begin{proof} We first show that $ (1)\Longrightarrow (2) \Longrightarrow (3) \Longrightarrow(4).$  Clearly (2) is immediate from the definition of strong coincidence. By Theorem~\ref{Berg} we have that (2) implies (3) and hence (4).
In what follows we will show that $ (4)\Longrightarrow  (1).$ To this end, we introduce the machinery of ``strand space" (a convenient presentation of tiling space) which will allow us to apply results developed for the $\R$-action on strand space to the shift action on words.

Let $\tau$ be an irreducible primitive substitution of Pisot type on the alphabet $\mathcal{A}=\{1,\ldots,n\}$ with Abelianization $M,$  dilation $\lambda$  and normalized positive right eigenvector ${\bf w}=(w_1,\ldots,w_n)^t$: $$M{\bf w}=\lambda{\bf w} ; \,\,\,||{\bf w}||=\sqrt{w_1^2+\cdots+w_n^2}=1.$$
We denote by $(X,T)$ the subshift of $\mathcal{A}^\N$ obtained by taking the shift orbit closure of a right-infinite $\tau$-periodic word $x$ and by $(\bar{X},T)$ the subshift of
$\mathcal{A}^{\Z}$ on the shift orbit closure of a bi-infinite $\tau$-periodic word $\bar{x}$.
(We require that $\bar{x}_{-1}\bar{x}_0$ be a factor of $x$.)

From the irreducible Pisot hypothesis, the spectrum of $M$ is nonzero and disjoint from the unit circle. Thus $M$ is a linear isomorphism and $\R^n$ decomposes as an invariant direct sum $$\R^n=E^u\oplus E^s$$
with $E^u=\{t{\bf w}:t\in\R\}$ and $M^k{\bf v}\to0$ as $k\to\infty$ for all ${\bf v}\in E^s$. Let $pr^u:\R^n\to E^u$ and $pr^s:\R^n\to E^s$ denote the corresponding projections. We denote by ${\bf e_i}$ the standard unit vector ${\bf e_i}=(0,\ldots,1,\ldots,0)^t$ and by $\sigma_i$ the arc $\sigma_i=\{t{\bf e_i}:0\le t\le1\}$, $i=1,\ldots,n$. By a {\em segment},
we will mean an arc of the form ${\bf v}+\sigma_i$ with ${\bf v}\in\R^n$: $i$ is its {\em type}, ${\bf v}$ is its {\em initial vertex}, and ${\bf v}+{\bf e_i}$ is its {\em terminal vertex}.
A {\em strand} is a collection of segments $\{{\bf v_i}+\sigma_{x_i}\}_{i\in \Z}$ with the property that the terminal vertex of ${\bf v_i}+\sigma_{x_i}$ equals the initial vertex of ${\bf v_{i+1}}+\sigma_{x_{i+1}}$ for all $i\in\Z$. Such a strand {\em follows pattern}
$x=(x_{i})\in\mathcal{A}^\Z$. (Note that a strand is a collection, rather than a sequence, so if a strand follows pattern $x$, it also follows pattern $T^k(x)$ for all $k\in\Z$.) Let $\mathcal{S}$ denote the collection of all strands in $\R^n$.

Given a segment $\sigma={\bf v}+\sigma_i$, let $$\Sigma_{\tau}(\sigma)=\{M{\bf v}+\sigma_{i_1},M{\bf v}+{\bf e_{i_1}}+\sigma_{i_2},\ldots,M{\bf v}+{\bf e_{i_1}}+\cdots+{\bf e_{i_{l-1}}}+\sigma_{i_l}\},$$
provided $\tau(i)=i_1\ldots i_l$, and define $\Sigma_{\tau}:\mathcal{S}\to\mathcal{S}$ by
$$\Sigma_{\tau}(S)=\cup_{\sigma\in S}\Sigma_{\tau}(\sigma).$$
Note that if $S$ follows pattern $x$, then $\Sigma_{\tau}(S)$ follows pattern $\tau(x)$.
For $R>0$, let $\mathcal{C}^R=\{{\bf v}\in\R^n:||pr^s({\bf v})||<R\}$ denote the $R$-cylinder about $E^u$ and let $\mathcal{S}^R=\{S\in\mathcal{S}:\sigma\subset\mathcal{C}^R \text { for all }  \sigma\in S\}$ denote the collection of all strands in $\mathcal{C}^R$.

For the following, see Lemma 5.1 in \cite{BK1}.

\begin{lemma}\label{existence of R_0} There is $R_0$ so that $\Sigma_{\tau}(\mathcal{S}^R)\subset \mathcal{S}^R$ for all $R\ge R_0$. Furthermore, if $S\in\mathcal{S}^R$ for some $R$, then there is $k\in\N$ so that $\Sigma_{\tau}^k(S)\in\mathcal{S}^{R_0}$.
\end{lemma}

With $R_0$ as in the lemma, the {\em strand space associated with $\tau$} is the global attractor
$$\mathcal{S}^*_{\tau}=\cap_{k\ge0}\Sigma_{\tau}^k(\mathcal{S}^{R_0}).$$
There is a natural metric topology on $\mathcal{S}^*_{\tau}$ in which stands $S$ and $S'$ are close if there is ${\bf v}\in\R^n$ and $r>0$, $||{\bf v}||$ small and $r$ large, so that $\sigma\in S$ and $pr^u(\sigma)\subset \{t{\bf w}:|t|\le r\}\implies {\bf v}+\sigma\in S'$. That is, $S$ and $S'$ are close if, after a small translation, $S$ and $S'$ agree in a large neighborhood of the origin. Let $d$ denote a metric inducing this topology. There is a continuous $\R$-action on $\mathcal{S}^*_{\tau}$ given by
$$(S,t)\mapsto S-t{\bf w}=\{\sigma-t{\bf w}:\sigma\in S\}.$$
This $\R$-action may not be minimal. To clean things up a bit, we define the {\em minimal strand space associated with $\tau$}, $\mathcal{S}_{\tau}$, to be the $\omega$-limit set of any $S\in\mathcal{S}^*_{\tau}$:
$$\mathcal{S}_{\tau}=\cap_{T\ge0}cl\{S-t{\bf w}:t\ge T\}.$$
The resulting space does not depend on which $S$ is chosen, and the $\R$-action on $\mathcal{S}_{\tau}$ is minimal.
Furthermore, $\Sigma_{\tau}$ restricted to $\mathcal{S}_{\tau}$ is a homeomorphism and the dynamics interact by
$$\Sigma_{\tau}(S-t{\bf w})=\Sigma_{\tau}(S)-\lambda t{\bf w}.$$
Every strand in the minimal strand space follows the pattern of some word in $\bar{X}$.

Strands $S,S'\in\mathcal{S}_{\tau}$ are {\em proximal} (with respect to the $\R$-action), denoted $S\sim_p S'$,  if
$\inf_{t\in\R}d(S-t{\bf w},S'-t{\bf w})=0$.

\begin{lemma}\label{proximality is finite} There is $B\in\N$ so that $\sharp\{S'\in\mathcal{S}_{\tau}:S'\sim_p S\}\le B$, for all $S\in\mathcal{S}_{\tau}$. If $\lim_{i\to\infty}d(S-t_i{\bf w},S'-t_i{\bf w})=0$, then there is $r_i\to\infty$ so that $S-t_i{\bf w}$ and $S'-t_i{\bf w}$ have exactly the same segments in the $r_i$-ball centered at 0.
\end{lemma}
\begin{proof} Theorem 4.2 of \cite{BBK} asserts that the map $g:\mathcal{S}_{\tau}\to X_{max}$ onto the maximal equicontinuous factor of the $\R$-action on $\mathcal{S}_{\tau}$ is boundedly finite-to-one. As proximallity is trivial for equicontinuous actions,
$\{S'\in\mathcal{S}_{\tau}:S'\sim_p S\}\subset g^{-1}(g(S))$.

The second assertion follows from \cite{BKel}: $\mathcal{S}_{\tau}$ is a Pisot family substitution tiling space and in all such spaces proximality and strong proximality (the condition that the strands exactly agree in large balls)  are the same relation.
\end{proof}

\begin{lemma}\label{recognizability} Suppose that $S\in\mathcal{S}$ follows pattern $x\in\bar{X}$. Then there is a unique $S'\in\mathcal{S}$ such that $\Sigma_{\tau}(S')=S$ and $S'$ follows a pattern in $\bar{X}$.
\end{lemma}
\begin{proof} It follows from Moss\'{e}'s recognizability result (\cite{M}), that there is $x'\in\bar{X}$ so that $\tau(x')=T^k(x)$ for some $k\in\Z$, and such an $x'$ is unique, up to shift. Let $S_1$ be any strand that follows pattern $x'$. Then $\Sigma_{\tau}(S_1)$ follows pattern $x$ so there is ${\bf v_1}\in\R^n$ so that $\Sigma_{\tau}(S_1)=S+{\bf v_1}$:
let $S'=S_1-M^{-1}{\bf v_1}$. If $S''$ is a strand that follows a pattern in $\bar{X}$ and $\Sigma_{\tau}(S'')=S$, that pattern must be (a shift of) $x'$, so $S''=S'+{\bf v}$ for some ${\bf v}\in\R^n$. Then $S=S+M{\bf v}$, and it must be that $M{\bf v}=0$, since $x$ is not $T$-periodic (see \cite{HZ}). Thus ${\bf v}=0$ and $S'$ is unique.
\end{proof}

\begin{lemma}\label{S_x} Given $x=(x_i)\in \bar{X}$ there is a unique $S=S_x=\{{\bf v_i}+\sigma_{x_i}\}_{i\in\Z}\in\mathcal{S}_{\tau}$ with the properties: $S$ follows pattern $x$; and
${\bf v_0}\in E^s$.
\end{lemma}
\begin{proof} Let $y=(y_i)\in\bar{X}$ be $\tau$-periodic, say $\tau^m(y)=y$ with $m>0$. Let $S_y$ be the strand $S_y=\{\ldots, -{\bf e_{y_{-2}}}-{\bf e_{y_{-1}}}+\sigma_{y_{-2}},
-{\bf e_{y_{-1}}}+\sigma_{y_{-1}}\}\cup\{\sigma_{y_0},{\bf e_{y_0}}+\sigma_{y_1}, {\bf e_{y_0}}+{\bf e_{y_1}}+\sigma_{y_2},\ldots\}$. Then $\Sigma_{\tau}^{km}(S_y)=S_y$ for all $k\in\N$, so $S_y\in\mathcal{S}_{\tau}$, and $S_y$ follows pattern $y$. Since $\bar{X}$ is minimal under the shift, there are $k_i\in\N$ with $T^{k_i}(y)\to x$. Let
$t_i=\sum_{j=0}^{n_i-1}w_{y_j}$. Then
$S_y-t_i{\bf w}\to S_x\in\mathcal{S}_{\tau}$ with $S_x$ as desired.

Suppose that $S'_x\in\mathcal{S}_{\tau}$ also follows pattern $x$ and has initial vertex on $E^s$. Then $S'_x=S_x+{\bf v}$ for some ${\bf v}\in E^s$. It follows from Lemma \ref{S_x} that, for each $k\in\N$, $\Sigma_{\tau}^{-k}(S'_x)=\Sigma_{\tau}^{-k}(S_x)+M^{-k}{\bf v}\in \mathcal{S}_{\tau}$. If ${\bf v}$ is not 0, there is $k$ large enough so that $||M^{-k}{\bf v}||> 2R_0$. But then $\Sigma_{\tau}^{-k}(S'_x)$ and $\Sigma_{\tau}^{-k}(S_x)$ can't both be in $\mathcal{S}_{\tau}\subset \mathcal{S}^{R_0}$. Thus $S_x$ is unique.
\end{proof}

\begin{lemma}\label{finite differences} There is $K\in\N$ with the property: if $x=(x_i),x'=(x'_i)\in\bar{X}$ and $S_x,S_{x'}\in\mathcal{S}_{\tau}$ are as in Lemma \ref{S_x} with initial vertices ${\bf v_0},{\bf v'_0}\in E^s$, and $k$th vertices ${\bf v_k}={\bf v_0}+\sum_{j=0}^{k-1}{\bf e_{x_j}},{\bf v'_k}={\bf v'_0}+\sum_{j=0}^{k-1}{\bf e_{x'_j}}$, $k\in\N$, then $\sharp\{{\bf v_k}-{\bf v'_k}:k\in\N\}\le K$, and, furthermore, $\sharp\{pr^u({\bf v_k}-{\bf v'_k}):k\in\N,x,x'\in X\}\le K$, where in this last bound the $x,x'$ vary over all of $X$.
\end{lemma}
\begin{proof} Let $P=\max_{i\in\{1,\ldots,n\}}||pr^s({\bf z})||$. The strands $S_x-{\bf v_0}$ and $S_{x'}-{\bf v'_0}$ are in $\mathcal{S}^{2R_0}$ and ${\bf v_k}-{\bf v_0},{\bf v'_k}-{\bf v'_0}\in\Z^n\cap\{{\bf z}=(z_1,\ldots,z_n)^t:\sum_{i=1}^nz_i=k\}$. Let ${\bf \bar{ z}_k}$ be the closest point in $\Z^n\cap\{{\bf z}=(z_1,\ldots,z_n)^t:\sum_{i=1}^nz_i=k\}$ to $E^u$, Then $||pr^s({\bf \bar{z}_k})||\le P$,  ${\bf v_k}-{\bf v_0}-{\bf \bar{z}_k},{\bf v'_k}-{\bf v'_0}-{\bf \bar{z}_k}\in \{{\bf z}=(z_1,\ldots,z_n)^t\in\Z^n:\sum_{i=1}^nz_i=0, ||{\bf z}||\le 2R_0+P\}$, and $pr^u({\bf v_k}-{\bf v'_k})=pr^u(({\bf v_k}-{\bf v_0}-{\bf \bar{z}_k})-({\bf v'_k}-{\bf v'_0}-{\bf \bar{z}_k} ))$. Thus $K=(\sharp\{{\bf z}=(z_1,\ldots,z_n)^t\in\Z^n:\sum_{i=1}^nz_i=0, ||{\bf z}||\le 2R_0+P\})^2$ will work.
\end{proof}

\begin{proposition}\label{shift proximality is finite} Let $x,y\in X$. Then:
\begin{enumerate}
\item $\{x'\in X:x' \text{ is proximal with } x\}$ is finite; and
\item if $x$ and $y$ are proximal and fixed by $\tau$, then $x$ and $y$ are strongly coincident.
\end{enumerate}
\end{proposition}
\begin{proof}
Suppose $x'$ is proximal with $x$. Extend $x$ and $x'$ to bi-infinite words in $\bar{X}$ - call these extensions $x$ and $x'$. There is $k_i\to\infty$ so that the $j$th coordinates
of $T^{k_i}(x)$ and $T^{k_i}(x')$ are the same for all $|j|\le i$.
Let $S_x,S_{x'}\in \mathcal{S}_{\tau}$  be as in Lemma \ref{S_x}. By Lemma \ref{finite differences}, there is a subsequence $k_{i_j}$ so that ${\bf v_{k_{i_j}}}-{\bf v'_{k_{i_j}}}\equiv {\bf v}$ is constant (here ${\bf v_k}, {\bf v'_k}$ denote the $k$th vertices of $S_x, S_{x'}$). Let $pr^u({\bf v_{k_{i_j}}})=t_j{\bf w}$ and $pr^u({\bf v})=t{\bf w}$. After passing to a subsequence, we may assume that $S_x-t_j{\bf w}\to S\in\mathcal{S}_{\tau}$ and
$S_{x'}-(t_i-t){\bf w}\to S'\in\mathcal{S}_{\tau}$. Then there is $y\in \bar{X}$ so that $S$ and $S'$ both follow pattern $y$ and have initial vertices on $E^s$. By the uniqueness in Lemma \ref{S_x}, these initial vertices, which differ by $pr^s({\bf v})$, must be the same. Thus, ${\bf v}=pr^u({\bf v})$, and $S_x\sim_pS_{x'}+{\bf v}$ in $\mathcal{S}_{\tau}$. By Lemma \ref{S_x}, there are only finitely many ${\bf v}$ that can arise this way, and by Lemma \ref{proximality is finite}, there are only finitely many strands in $\mathcal{S}_{\tau}$ proximal with $S_x$. Thus, the set $\{x'\in X:x' \text{ is proximal with }x\}$ is also finite.

For (2), we may take extensions of $x$ and $y$ in $\bar{X}$ that are $\tau$-periodic: say $\tau^m(x)=x$ and $\tau^m(y)=y$ with $m>0$. The strands $S_x$ and $S_y$ in $\mathcal{S}_{\tau}$ have initial vertices at the origin, are fixed by $\Sigma_{\tau}^m$, and, by the above, there is $t$ so that $S_x\sim_pS_y+t{\bf w}$. Then $S_x=\Sigma_{\tau}^{km}(S_x)\sim_p\Sigma_{\tau}^{km}(S_y+t{\bf w})=\Sigma_{\tau}^{km}(S_y)+\lambda^{km}t{\bf w}=S_y+\lambda^{km}t{\bf w}$. Since there are only finitely many strands proximal with $S_x$ (and $y$ is not $T$-periodic, again by \cite{HZ}), $t=0$. Thus $S_x\sim_p S_y$ and there is $k\in\N$ so that $S_x$ and $S_y$ not only have the same $k$th vertex, but also share their $k$th segment (Lemma \ref{proximality is finite}). That is, $x$ and $y$ are strongly coincident.
\end{proof}

We return to the proof that $ (4)\Longrightarrow  (1)$ in Theorem
\ref{SCCP2}. Let $x$ and $y$ be fixed points of $\tau$ and suppose
that for every prefix $u$ of $y$ the set $x\big|_u$ is a central
set. This means that for every prefix $u$ of $y$ the set
$x\big|_u$ belongs to some minimal idempotent $p_u \in \beta
\nats.$  The collection \[{\mathcal{P}}=\{p_u^*(x)|\,
u\,\,\mbox{is a prefix of}\,\,y\}\] consists of infinite words in
$X$ each proximal to $x.$  By (1) of Proposition~\ref{shift
proximality is finite}, the set ${\mathcal{P}}$ is finite.
Moreover since $p_u^*(x)\rightarrow y$ as $|u|\rightarrow +\infty$
(since $u$ is a prefix of $p_u^*(x)),$ it follows that $y\in
{\mathcal{P}}$ and hence $y$ is proximal to $x.$ Whence by (2) of
 Proposition~\ref{shift proximality is finite} we deduce that $x$ and $y$ are strongly coincident.
\end{proof}

\noindent We now recall the following result from \cite{BPZ}:

\begin{theorem}\label{same} Let $\omega \in {\mathcal A}^{\nats}$ be uniformly recurrent. Then the set
$\omega\big|_u$ is an IP-set if and only if it is a central set.
\end{theorem}

\noindent Combining theorems~\ref{SCCP2} and \ref{same} we obtain

\begin{corollary} Let $\tau$ be an irreducible primitive substitution of Pisot type. Then for any pair of fixed points $x$ and $y$ of $\tau$ the following are equivalent, 
$x$ and $y$ are strongly coincident.
if and only if for any prefix $u$ of $y,$ the set $x\big|_u$ is an IP-set.
\end{corollary}

\section{Abstract numeration systems}

Here we present an alternative and constructive proof that for a general primitive substitution $\tau$  (not necessarily irreducible and of Pisot type) if two fixed points $x$ and $y$ of $\tau$ are strongly coincident, then for every prefix $u$ of $y$ the set $x\big|_u$ is a central set. We make use of the  Dumont-Thomas numeration systems defined by substitutions \cite{DT1,DT2}. Since in the irreducible Pisot case, condition  (4) alone implies the strong coincidence condition, this method of proof may provide a new insight to the strong coincidence conjecture. We begin with a brief review of these numerations systems.

Let $\tau$ denote a substitution on a finite alphabet $\mathcal{A}.$ For simplicity we assume that $\tau$ has at least one fixed point $x=x_0x_1x_2\ldots$ beginning in some letter $a\in \mathcal{A}.$
The idea behind the numeration system is quite natural: every coordinate $x_n$ of the fixed point $x$ is in the image of $\tau$ of some coordinate $x_m$ with $m\leq n.$ More precisely, consider the least positive integer $m$ such that $x_0x_1\ldots x_n$ is a prefix of $\tau(x_0x_1\ldots x_m).$  In this case we can write $x_0x_1\ldots x_n=\tau(x_0x_1\ldots x_{m-1})u_nx_n$ where $u_nx_n$ is a  prefix of $\tau(x_m).$ We now imagine a directed arc  from $x_m$ to $x_n$ labeled $u_n.$ In this way every coordinate $x_n$ is the target of exactly one arc, and the source of $|\tau(x_n)|$-many arcs. It follows that for each $n$ there is a unique path $s$ from $x_0$ to $x_n.$ Thus every natural number $n$ may be represented by a finite sequence of labels $u_i$ obtained by reading the labels along the path $s$ in the direction from $x_0$ to $x_n.$

More formally,  associated to $\tau$ is a directed graph  $\mathcal{G}(\tau)$ defined as follows: the vertex set of $\mathcal{G}(\tau)$ is the set $\mathcal{A}.$ Given any pair of vertices $a,b$ we draw a directed edge from $a$ to $b$ labeled $u\in \mathcal{A}^*$ if $ub$ is a prefix of $\tau(a).$
In other words, for every occurrence of $b$ in $\tau(a)$ there is a directed edge from $a$ to $b$ labeled by the prefix (possibly empty) of $\tau(a)$ preceding the given occurrence of $b.$ Figure 1 depicts the graph $\mathcal{G}(\tau)$ for the Fibonacci substitution $a\mapsto ab,$ $b\mapsto a.$

\begin{figure}[h]
  \begin{center}
    \unitlength=4pt
    \begin{picture}(25, 10)(0,-5)
      \gasset{Nw=5,Nh=5,Nmr=2.5,curvedepth=4}
      \thinlines
      \node[](A1)(0,0){$a$}
      \drawloop[loopangle=180](A1){$\varepsilon$}
      \node[](A2)(25,0){$b$}

      \drawedge(A1,A2){$a$}
      \drawedge(A2,A1){$\varepsilon$}
    \end{picture}
  \end{center}
  \caption{The Fibonacci automaton}
\end{figure}
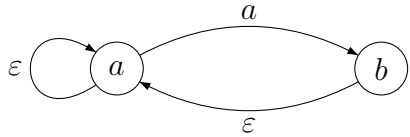\label{GF}

For simplicity, in case some letter $b$ occurs multiple times in $\tau(a),$ we draw just one directed edge from $a$ to $b$ having multiple labels as described above. This is shown in Figure 2 in the case of the substitution $a\mapsto aab,$ $b\mapsto bbaab.$

\begin{figure}[h]
  \begin{center}
    \unitlength=4pt
    \begin{picture}(25, 10)(0,-5)
      \gasset{Nw=5,Nh=5,Nmr=2.5,curvedepth=4}
      \thinlines
      \node[](A1)(0,0){$a$}
      \drawloop[loopangle=180](A1){$\varepsilon\,,a$}
      \node[](A2)(25,0){$b$}
      \drawloop[loopangle=0](A2){$\varepsilon\,,b\,,bbaa$}
      \drawedge(A1,A2){$aa$}
      \drawedge(A2,A1){$bb\,,bba$}
    \end{picture}
  \end{center}
  \caption{The automaton of $a\mapsto aab,$ $b\mapsto bbaab.$}
\end{figure}
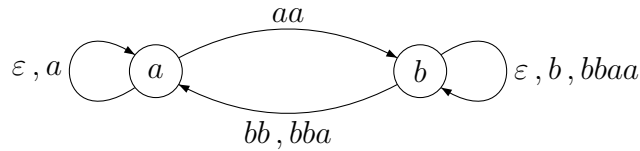\label{Gtau}

Let $x=x_0x_1x_2\ldots$ denote the fixed point of $\tau$ beginning in $a.$  Then the graph $\mathcal{G}(\tau)$ has a singleton loop based at $a$ labeled with the empty word $\varepsilon.$  We consider this to be the empty or $0$th path at $a.$ More generally by a path at $a\in \mathcal{A}$ we mean a finite sequence of edge labels $u_0u_1u_2\cdots u_n$ corresponding to a path in $\mathcal{G}(\tau)$ originating at vertex $a$ with the condition that $u_0\neq \varepsilon$ whenever the length of the path $n>0.$
For example in the case of the Fibonacci substitution,  except for the path $s=\varepsilon,$  each path is given by a word in $\{a, \varepsilon\}$ beginning in $a$ and not containing the factor $aa.$
For each path $s=u_0u_1u_2\cdots u_n$ set \[\rho(s)=\tau^n(u_0)\tau^{n-1}(u_1)\tau^{n-2}(u_2)\cdots \tau(u_{n-1})u_n\] and $\lambda(s)=|\rho(s)|.$
In \cite{DT1,DT2} it is shown that for each path
$s$ at $a,$ the word $\rho(s)$ is a prefix of the fixed point $x$ at $a$ and conversely for each prefix $u$ of $x$ there is a unique path $s$ at $a$ with $\rho(s)=u.$ This correspondence defines a numeration system in which every natural number $l$ is represented by the path $s=u_0u_1u_2\cdots u_n$  in $\mathcal{G}(\tau)$ from vertex $a$ to vertex $x_l$ corresponding to the prefix of length $l$ of $x,$ so that
\begin{equation*}(*)\,\,\,\,\,\,l= \lambda(s)= |\tau^n(u_0)|+|\tau^{n-1}(u_1)|+|\tau^{n-2}(u_2)|+\cdots +|\tau(u_{n-1})|+|u_n|.\end{equation*}

Generally by the numeration system one means the quantities $|\tau^n(u)|$ for all $n\geq 0$ and all proper prefixes $u$ of the images under $\tau$ of the letters of $\mathcal{A}.$ Then a {\it proper} representation of $l$ in this
numeration is an expression of the form (*) corresponding to a path  $s=u_0u_1u_2\cdots u_n$  in $\mathcal{G}(\tau).$

In the case of a uniform substitution  of length $k$ this corresponds to the usual base $k$-expansion of $l.$
In the case of the Fibonacci substitution, each $u_n\in \{\varepsilon, a\}$ and $u_iu_{i+1}\neq aa$ for each $0\leq i\leq n-1.$ Thus this representation of $l$ is the so-called Zeckendorff representation of $l$ in which $l$ is expressed as a sum of distinct Fibonacci numbers via the greedy algorithm (see \cite{Zeck}).

In general, this numeration system not only depends on the substitution $\tau$ but also on the choice of fixed point. For example for the substitution in Figure 2 the number $5$ is represented by the path $a,aa$ from vertex $a$ or by the path $b,\varepsilon$ from vertex $b.$
In fact, $\tau(a)aa=aabaa$ is the prefix of length $5$ of $\tau^\infty(a)$ while $\tau(b)\varepsilon=bbaab$ is the prefix of length $5$ of $\tau^\infty (b).$

An alternative reformulation is as follows:  Given two distinct paths $s=u_0u_1u_2\cdots u_n$ and $t=v_0v_1v_2\cdots v_m$ both starting from the same vertex $a,$ we write $s<t$ if either $n<m$ or if $n=m$  there exists $i \in \{0,1,\ldots ,n\}$ such that $u_j=v_j$ for $j<i,$ and $|u_i|<|v_i|.$ This defines a total order on the set of all paths starting from vertex $a.$ In the case of the Fibonacci substitution, we list the paths at $a$  in increasing order
\[\varepsilon, a, a\varepsilon, a\varepsilon\varepsilon, a\varepsilon a, a\varepsilon \varepsilon \varepsilon, a \varepsilon \varepsilon a, a \varepsilon a \varepsilon, a \varepsilon \varepsilon \varepsilon \varepsilon, \ldots\]
Thus there is an order preserving correspondence between $0,1,2,3,\ldots$ and the set of all paths at $a$ ordered in increasing order.

While these numeration systems are very natural and simple to define, they are typically extremely difficult to work with in terms of addition and multiplication.

Let $a$ and $b$ be distinct vertices in $\mathcal{G}(\tau).$ We say a path $s$ originating at $a$ is {\it synchronizing} relative to $b$ if there exists a path $s'$ originating at $b$ having the same terminal vertex as $s$ and with $\lambda(s)=\lambda(s').$ From this point of view the strong coincidence conjecture
implies that \[\{\lambda (s)\,|\,s=\mbox{a synchronizing path relative to}\, b\}\] is a thick set.\\

Let $\tau$ be a primitive substitution satisfying the strong coincidence condition. Suppose $x$ and $y$ are fixed points of $\tau$ beginning in $a$ and $b$ respectively.
We will show that $x\big|_u$ is a central set for every prefix $u$ of $y.$
Since $x$ and $y$ are strongly coincident, we can write
$x=scx',$ and $y=tcy'$ for some $s,t \in {\mathcal A}^+,$ $c\in {\mathcal A},$  and $x',y' \in {\mathcal A}^\infty$ with $s\sim_{\mbox{ab}} t.$ By replacing $\tau$ by a sufficiently large power of $\tau,$ we can assume that
\begin{itemize}
\item $sc$ is a prefix of $\tau(a),$
\item $tc$ is a prefix of $\tau(b),$
\item $b$ occurs in $\tau(c).$
\end{itemize}

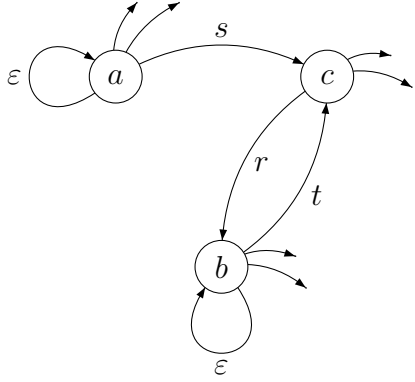
\begin{figure}[h]
  \begin{center}
    \unitlength=4pt
    \begin{picture}(15, 32)(0,-8)
    \gasset{Nw=5,Nh=5,Nmr=2.5,curvedepth=3}
    \thinlines
    \node[](A0)(10,0){$b$}
    \node[](A1)(0,18){$a$}
    \node[](A2)(20,18){$c$}
    \gasset{Nw=0,Nh=0,Nmr=0,curvedepth=3}
    \node[](A3)(26,20){$$}
    \node[](A4)(28,17){$$}
    \node[](A5)(2,25){$$}
    \node[](A6)(6,25){$$}
    \node[](A7)(17,1){$$}
    \node[](A8)(18,-2){$$}

    \drawloop[loopangle=270](A0){$\varepsilon$}
      \drawloop[loopangle=180](A1){$\varepsilon$}
    \gasset{curvedepth=-3}

    \gasset{curvedepth=+1}
    \drawedge(A2,A3){$$}
    \drawedge(A2,A4){$$}
     \drawedge(A1,A5){$$}
      \drawedge(A1,A6){$$}

      \drawedge(A0,A7){$$}
      \drawedge(A0,A8){$$}

    \gasset{curvedepth=+3}
    \drawedge(A1,A2){$s$}

 \gasset{curvedepth=-3}
 \drawedge(A2,A0){$r$}
    \drawedge[ELside=r](A0,A2){$t$}
    \end{picture}
  \end{center}
  \caption{Vertices $a,b,c$ of $\mathcal{G}(\tau)$}
\end{figure}\label{fig3}

 Thus in $\mathcal{G}(\tau)$ there is a directed edge from $a$ to $c$ labeled $s,$ a directed edge from $b$ to $c$ labeled $t,$ and a directed edge from $c$ to $b$ labeled $r$ for some prefix $r$ of $\tau(c).$ See Figure 3.

We now define a sequence of paths $(p_i)_{i\geq 0}$ from $a$ to $b$ by

 \[p_i=s,r, \underbrace{\varepsilon, \varepsilon, \ldots ,\varepsilon}_{2i}.\]
Put $n_i=\lambda (p_i).$ Then clearly
$\{n_i\,|\,i\geq 0\}\subseteq x\big|_b.$
We now show that any finite sum of distinct elements from the set $\{n_i\,|\,i\geq 0\}$ is contained
in $x\big|_b.$ Set
\[q_i=t,r, \underbrace{\varepsilon, \varepsilon, \ldots ,\varepsilon}_{2i}.\]
Then each $q_i$ is a path from $b$ to $b$ and since $s$ and $t$ are Abelian equivalent it follows that $\lambda(p_i)=\lambda(q_i).$
Fix $k\geq 1$ and choose $i_1< i_2<\cdots <i_k.$
Then
\begin{align*}\sum_{j=1}^k \lambda(p_{i_j})&= \lambda(p_{i_k}) +\sum_{j=1}^{k-1} \lambda(p_{i_j})\\&=\lambda(p_{i_k}) +\sum_{j=1}^{k-1} \lambda(q_{i_j})\\&=
|\tau^{2i_k+1}(s)| + |\tau^{2i_k}(r)| + \sum_{j=1}^{k-1} (|\tau^{2i_j+1}(t)|+|\tau^{2i_j}(r)|)\\&=
|\tau^{2i_k+1}(s)\tau^{2i_k}(r)\tau^{2i_{k-1}+1}(t)\tau^{2i_{k-1}}(r) \tau^{2i_{k-2}+1}(t)\tau^{2i_{k-2}}(r) \cdots \tau^{2i_{1}+1}(t)\tau^{2i_{1}}(r) |
\end{align*}
 which is represented by a path in $\mathcal{G}(\tau)$ from $a$ to $b$ and hence corresponds to an occurrence of $b$ in $x.$ This shows that $x\big|_b$ is an IP-set. It now follows from Theorem~\ref{same}  that $x\big|_b$ is a central set.
 A similar argument applies for any prefix $u$ of $y$ by defining the paths $p_i$ by
 \[p_i=s,r, \underbrace{\varepsilon, \varepsilon, \ldots ,\varepsilon}_{N_i}\]
 with $N_i$ sufficiently large.

\end{document}